\documentclass[a4paper, 12 pt, oneside, reqno]{amsart} 
\usepackage{amsfonts, amssymb, amsmath, eucal, amsthm}
\usepackage[hidelinks]{hyperref}

\newtheorem{lemma}{Lemma}
\numberwithin{lemma}{section}
\newtheorem{theorem}[lemma]{Theorem}
\newtheorem*{theorem*}{Theorem}
\newtheorem*{theorema}{Theorem A}
\newtheorem*{theoremb}{Theorem B}
\newtheorem*{theoremc}{Theorem C}
\newtheorem*{corollary*}{Corollary}
\newtheorem{proposition}[lemma]{Proposition}
\newtheorem{corollary}[lemma]{Corollary}

\theoremstyle{definition}
\newtheorem{definition}[lemma]{Definition}

\newtheorem{example}[lemma]{Example}

\newtheorem*{remark*}{Remark}

\usepackage{geometry}

\newcommand{\K}{\mathcal{K}}
\newcommand{\C}{\mathcal{C}}
\newcommand{\CB}{\mathcal{CB}}

\newcommand{\bbZ}{\mathbb{Z}}
\newcommand{\bbR}{\mathbb{R}}
\newcommand{\bbC}{\mathbb{C}}

\title{Groups of convex bodies}
\author{Richard Hepworth}
\address{Institute of Mathematics\\
University of Aberdeen
}
\email{r.hepworth@abdn.ac.uk}

\subjclass[2020]{Primary 28A75, 52A20; Secondary 52B45, 19D99}
\keywords{Convex bodies, McMullen polynomiality, K-theory}

\begin{document}

\begin{abstract}
	In this paper we introduce and study a topological abelian group
	of convex bodies, analogous to the scissors congruence group
	and McMullen's polytope algebra,
	with the universal property that continuous valuations on convex
	bodies correspond to continuous homomorphisms on the group
	of convex bodies.
	To study this group, we first 
	obtain a version of McMullen polynomiality
	for valuations that take values not in fields or vector spaces,
	but in abelian groups.
	Using this, we are able to equip the group of convex bodies with
	a grading that consists of real vector spaces in all positive degrees,
	mirroring one of the main structural properties of the polytope algebra.
	It is hoped that this work can serve as the starting point for a
	$K$-theoretic interpretation of valuations on convex bodies.
\end{abstract}

\maketitle
%\setcounter{tocdepth}{1}
%\tableofcontents

\section{Introduction}

\subsection*{Valuations}
Let $V$ be a finite-dimensional real vector space. 
A \emph{convex body} in $V$
is a nonempty compact convex subset $X\subseteq V$. 
The set $\K(V)$ of all convex bodies in $V$ is topologised using
the Hausdorff metric.
A function $\varphi\colon\K(V)\to A$, with values in an abelian group
$A$, is a \emph{valuation} if it satisfies the relation
\[
	\varphi(B\cup C) = \varphi(B)+\varphi(C) - \varphi(B\cap C)
\]
whenever $B,C\in\K(V)$ with $B\cup C\in \K(V)$.  Continuous valuations on $\K(V)$ with values in $\bbR$ or $\bbC$
have been studied under many different symmetry conditions.
For example, Hadwiger's theorem classifies all continuous isometry-invariant
valuations on convex bodies in Euclidean space~\cite{Hadwiger,Chen,Klain},
Klain and Schneider classified all simple continuous translation-invariant
valuations~\cite{Klain,Schneider},
and Alesker classified continuous valuations on 
$\K(\bbC^n)$ that are invariant under $U(n)$ 
and translations~\cite{AleskerLefschetz}.
Many of the results on continuous valuations with values in $\bbR$ or $\bbC$
extend naturally to continuous valuations with values in topological vector
spaces (see for example section~6.4 of~\cite{SchneiderBook}).
It is natural to ask what happens when values are taken in 
an arbitrary Hausdorff topological abelian group.
The answer is that, assuming translation-invariance,
essentially nothing new happens:

\begin{theorema}
	Let $V$ be a real finite-dimensional vector space.
	Let $\varphi\colon\K(V)\to A$ be a continuous translation-invariant 
	valuation with values in a Hausdorff topological abelian group $A$.
	Let $c\in A$ denote the value of $\varphi$ on any one-point convex
	body.
	Then the valuation $\varphi-c\colon\K(V)\to A$ admits a factorisation 
	\[
		\K(V)\xrightarrow{\ f\ }\mathcal{V}\xrightarrow{\ g\ }A
	\]
	where 
	$\mathcal{V}$ is a Hausdorff topological vector space,
	$f$ is a continuous translation-invariant valuation,
	and $g$ is a continuous homomorphism of topological abelian
	groups.
\end{theorema}

The need to subtract a constant cannot be overcome, since the constant map
$\K(V)\to\bbZ$ with value $1$ is a continuous translation invariant valuation.

\subsection*{Groups of convex bodies}
Theorem~A is a byproduct of work 
inspired by the relationship between scissors congruence 
and algebraic $K$-theory.
\emph{Scissors congruence} is the study of 
polytopes modulo dissection and rearrangement.
The \emph{scissors congruence
group} is the free abelian group on polytopes modulo symmetry, 
in which a polytope is identified with the sum of the parts in
any of its dissections.
Zakharevich~\cite{ZakharevichI,ZakharevichII,ZakharevichIII,ZakharevichIV}
established a $K$-theoretic interpretation of the scissors congruence group,
and produced a sequence of groups 
--- the $K$-theory groups of a certain Waldhausen category ---
whose $0^\mathrm{th}$ term is precisely the scissors congruence group,
and whose higher terms encode information about symmetries of polytopes.
Is there a $K$-theoretic interpretation of valuations on convex bodies?
And, as an initial step, can an analogue of the scissors congruence group 
be established and understood for convex bodies?
Theorem~A follows from our progress on the initial step, as we now explain.

Let $V$ be a finite dimensional real vector space,
and let $G$ be a group of affine transformations of $V$. 
In this paper we introduce the \emph{group of convex bodies} $\CB(V,G)$,
which is the free Hausdorff topological abelian group 
on the space $\K(V)$,
modulo the closed subgroup generated by the elements of the form
\[
	[gX]-[X]\quad\text{and}\quad[B\cup C]-[B]-[C]+[B\cap C]
\]
for $g\in G$, $X\in\K(V)$, and 
$B,C\in\K(V)$ with $B\cup C\in\K(V)$.
It is the target of a continuous $G$-invariant valuation
$\Phi\colon\K(V)\to \CB(V,G)$ that is universal in the following sense: 
Any continuous $G$-invariant valuation
$\varphi\colon\K(V)\to A$
with values in a Hausdorff topological abelian group $A$ 
factors uniquely as
$\varphi = \bar\varphi\circ\Phi$
for some continuous homomorphism $\bar\varphi\colon\CB(V,G)\to A$.
So $\CB(V,G)$ is our analogue of the scissors congruence group.
We obtain the following fundamental structural result,
from which Theorem~A follows quickly:

\begin{theoremb}
	Let $V$ be a finite-dimensional real vector space
	of dimension $d$,
	and let $G$ be a group of affine linear transformations of $V$ that
	includes all translations.
	Then there is a direct sum decomposition
	of Hausdorff topological abelian groups
	\[
		\CB(V,G)
		\cong
		\CB_0(V,G)\oplus\cdots\oplus\CB_d(V,G)
	\]
	and a corresponding decomposition 
	\[
		\Phi = \Phi_0\oplus\cdots\oplus\Phi_d
	\]
	of $\Phi$ into continuous $G$-invariant valuations
	$\Phi_i\colon\K(V)\to\CB_i(V,G)$.
	The group $\CB_0(V,G)$ is a copy of $\bbZ$ generated by
	$\Phi_0(X)$ for any $X\in\K(V)$.
	And for each $1\leqslant i\leqslant d$,
	$\CB_i(V,G)$ admits the structure of an $\bbR$-vector space,
	with respect to which $\Phi_i(\lambda X) = \lambda^i\Phi_i(X)$ 
	for $\lambda\in[0,\infty)$ and $X\in\K(V)$.
\end{theoremb}

Theorem~B shows that 
the topological abelian \emph{group} $\CB(V,G)$ is almost identical to
the topological \emph{vector space} $\bbR\otimes_{\bbZ}\CB(V,G)$ ---
a single copy of $\bbZ$ in the former
is replaced by a copy of $\bbR$ in the latter. 
We anticipate that $\bbR\otimes_{\bbZ}\CB(V,G)$ is simply the
(appropriately topologised) 
dual of the Banach space $\mathrm{Val}_G(V)$ of $G$-invariant valuations on $V$,
so that the study of $\CB(V,G)$ boils down to the study of valuations
in $\bbR$ or $\bbC$, which are well-understood in many cases.
We hope that the results of this paper will open the door to
a $K$-theoretic interpretation of valuations on
convex bodies, analogous to the $K$-theoretic interpretation
of scissors congruence, so that $\CB(V,G)$ becomes the $0^\mathrm{th}$ term
in a sequence whose higher terms encode information
about symmetries of convex bodies.

\subsection*{McMullen polynomiality}

Let us explain the main technical result underpinning Theorem~B.
McMullen polynomiality~\cite{McMullen} shows that 
if $\varphi\colon\K(V)\to \mathcal{V}$ 
is a translation-invariant continuous valuation 
with values in a real Hausdorff topological vector space $\mathcal{V}$,
then $\varphi$ admits a decomposition $\varphi=\varphi_0+\cdots+\varphi_d$
where each $\varphi_i\colon\K(V)\to \mathcal{V}$ is homogeneous of degree $i$,
i.e.~$\varphi_i(\lambda X) = \lambda^i\varphi_i(X)$ 
for $\lambda\in[0,\infty)$ and $X\in\K(V)$.
(See Theorem~6.3.5 of~\cite{Schneider}.)
To prove Theorem~B we establish a form of McMullen polynomiality
for valuations with values in Hausdorff topological abelian groups.
The simplest version of our polynomiality result is the following.

\begin{theoremc}
	Let $V$ be a finite dimensional real vector space, 
	let $A$ be a Hausdorff topological abelian group,
	and let $\varphi\colon\K(V)\to A$ 
	be a continuous translation-invariant valuation.
	Then there are continuous semigroup
	homomorphisms
	\[f_1,\ldots,f_d\colon [0,\infty)\to A\] 
	and a constant $f_0\in A$, all uniquely determined, such that
	\[
		\varphi(\lambda X) = f_0+f_1(\lambda^1)+\cdots+f_d(\lambda^d)
	\]
	for all $\lambda$.
\end{theoremc}

Applying this theorem to the universal valuation $\Phi\colon\K(V)\to\CB(V,G)$
and setting $\lambda=1$ gives us the splitting of $\Phi(X)$ into its components 
$\Phi_0(X)+\cdots+\Phi_d(X)$ as in Theorem~B, and Theorem~B 
in its full strength follows from a parameterised version of Theorem~C.

The difficulty in establishing Theorem~C
is that the usual definition of homogeneity cannot be used since the
target is not an $\bbR$-vector space.
We tackle this by using a notion of polynomiality,
defined in terms of iterated differences,
that applies to functions from abelian semigroups to abelian groups;
this method of recognising polynomiality goes back at least to
Fr\'echet~\cite{FrechetI,FrechetII}.

\subsection*{Comparison with the polytope algebra}

In~\cite{McMullenPolytope}, McMullen introduced the \emph{polytope algebra}
$\Pi(V)$ of a vector space $V$.
This is the free abelian group on the convex polytopes in $V$ modulo
translation and inclusion-exclusion relations,
made into an algebra under Minkowski addition.
It is closely related to the translation-invariant
convex body group $\CB(V,V)$, 
and indeed there is a canonical homomorphism of abelian groups
$\Pi(V)\to\CB(V,V)$ with dense image.

This paper can therefore be regarded as a 
first step in producing a ``convex body algebra'' analogous to $\Pi(V)$,
so that Theorem~B is then a partial analogue of one of the main structural
results on $\Pi(V)$~\cite[Theorem~1]{McMullenPolytope}.
The next step would be to consider the algebra structure that Minkowski sum
induces on $\CB(V,V)$.
Tensor products can be technically challenging in the topological setting,
and we expect that the convex body algebra will need to be placed
in a carefully chosen monoidal category of topological abelian groups
or vector spaces, such as stereotype spaces or condensed sets,
before an adequate theory can be established.

Comparing our methods with McMullen's, 
it seems that our approach of studying groups of convex bodies
by directly establishing 
an abelian group-valued version of McMullen polynomiality has no parallel 
in~\cite{McMullenPolytope},
and of course there are the technicalities arising from the 
fact that the $\CB(V,G)$ must be topologised while $\Pi(V)$ is not.

\subsection*{Outline of the paper}
Section~\ref{section-TAG} introduces topological abelian groups and 
the relevant facts and constructions we will need there.
Section~\ref{section-general-polynomiality} establishes the necessary background
material on polynomiality of maps with values in abelian groups,
and then section~\ref{section-McMullen} applies this to obtain a form of
McMullen polynomiality in this context, in particular proving
Theorem~C.
Section~\ref{section-grothendieck} introduces groups of convex bodies,
and finally section~\ref{section-structure} uses polynomiality to investigate
their structure, proving Theorems~A and~B.

\section{Topological abelian groups}\label{section-TAG}

In this brief section we will recall some necessary background on 
topological abelian groups.
For further reading we recommend Chapter~1 of Morris's book~\cite{Morris}
for generalities, 
and the introduction and section~1 of Thomas's paper~\cite{Thomas}
for material on free topological abelian groups.

Recall that a \emph{topological abelian group} 
$A$ is an abelian group equipped with a topology
such that addition and negation in $A$ define continuous maps 
$A\times A\to A$ and $A\to A$ respectively.
A topological abelian group $A$ is Hausdorff if and only if
$\{0\}\subseteq A$ is closed~\cite[p.5]{Morris}.
We will need to work with Hausdorff topological abelian groups in order
to use the uniqueness of limits.
We will abbreviate `topological abelian group' as TAG and
`Hausdorff topological abelian group' as HTAG.

For every TAG $A$ there is an HTAG $A/\overline{\{0_A\}}$, 
that we call its \emph{reflection};
the reflection is equipped with a canonical map $A\to A/\overline{\{0_A\}}$,
and any other continuous homomorphism $A\to B$ into an HTAG
factors uniquely as $A\to A/\overline{\{0_A\}}\to B$.
In practice this means that we can perform various constructions
(in particular, quotients and free HTAGs)
by first constructing the relevant TAG and then taking its reflection.
(The word reflection here is taken from the fact that the category
of HTAGs is a reflexive subcategory of the category of TAGs.)

Observe that if we have TAGs $A\supseteq B\supseteq C$,
then in the quotient group $A/C$ we have
$\overline{B/C} = \bar B/C$.
So in particular, given $A\supseteq B$, we have
$\overline{\{0_{A/B}\}}=\overline{B/B}=\bar B / B$, so that
$A/B$ is Hausdorff if and only if $B$ is closed.

We must now discuss free topological abelian groups.
This is a well-studied subject with many intricacies, 
but we only need the very basics.
The introduction and section~1 of~\cite{Thomas} provide us with
the existence theorem that we will need.

\begin{definition}[Free Hausdorff topological abelian group]\label{free-HTAG}
	Let $X$ be a topological space.
	The \emph{free Hausdorff topological abelian group on $X$},
	denoted $\bbZ X$, is an HTAG equipped with a continuous map
	$X\to\bbZ X$, with the following property:
	For any HTAG $A$, the map
	\[
		(\bbZ X\xrightarrow{\varphi} A)
		\longmapsto
		(X\to \bbZ X\xrightarrow{\varphi} A)
	\]
	is a bijection
	between the set of continuous homomorphisms $\bbZ X\to A$
	and the set of continuous maps $X\to A$.
	Given $x\in X$, we write $[x]\in\bbZ X$ for the image
	of $x$ under $X\to\bbZ X$.
\end{definition}

\begin{theorem}
	For any topological space $X$, the free Hausdorff topological
	abelian group $X\to\bbZ X$ exists.
\end{theorem}

For a proof, see Theorem~1.4 of~\cite{Thomas} and the paragraph that
follows it.
This constructs the free TAG $A(X)$, and then the free HTAG 
$\bbZ X$ is the reflection $A(X)/\overline{\{0_{A(X)}\}}$
of $A(X)$.
A brief alternative way to describe $A(X)$ is that it is obtained from the free
abelian group $FA$ on $X$ by considering all group
topologies for which the map $X\to FA$ is continuous, and then taking
the finest topology containing all of these.

In what follows we will have occasion to use `parameterised'
homomorphisms of HTAGs.
To achieve this we will make use of mapping spaces.
Recall that if $X$ and $Y$ are topological spaces, 
then the space $\C(X,Y)$ of continuous maps $X\to Y$ 
can be equipped with the \emph{compact-open topology}~\cite[p.285]{Munkres}.
This is the topology generated by the sets $S(C,U)=\{f\colon X\to Y\mid
f(C)\subseteq U\}$ for $C\subseteq X$ compact and $U\subseteq Y$ open.
Here are some basic properties of the compact-open topology:
\begin{enumerate}
	\item\label{co-hausdorff}
	If $Y$ is Hausdorff then so is $\C(X,Y)$.
	(See exercise~6 of~\cite[Ch.7]{Munkres}.)
	\item\label{co-products}
	If $X,Y_1,Y_2$ are topological spaces, then the map
	$\C(X,Y_1)\times\C(X,Y_2)\to \C(X,Y_1\times Y_2)$,
	$(f,g)\mapsto f\times g$, is continuous.
	(This is a simple application of the definition.)
	\item\label{co-pre-post}
	If $X',X,Y,Y'$ are topological spaces and $\alpha\colon X'\to X$
	and $\beta\colon Y\to Y'$ are continuous, then the map
	$\C(X,Y)\to\C(X',Y')$, $f\mapsto \beta\circ f\circ\alpha$
	is also continuous.
	(Again, this is a simple application of the definition.)
	\item\label{co-htag}
	Let $A$ be a Hausdorff topological abelian group,
	and let $X$ be a topological space.
	Then $\C(X,A)$, equipped with the operation of pointwise
	addition, is a Hausdorff topological abelian group.
	(This is proved by combining points~\ref{co-hausdorff},
	\ref{co-products} and~\ref{co-pre-post} with the fact that
	the structure maps of $A$ are continuous.)
	\item\label{co-curry}
	If $X,Y,Z$ are topological spaces with $X$ locally compact Hausdorff,
	then there is a bijection between continuous maps
	$X\times Y\to Z$ and continuous maps $Y\to\C(X,Z)$.
	The bijection sends $f\colon X\times Y\to Z$ to the map 
	$g\colon Y\to\C(X,Z)$ defined by $g(y)(x)=f(x,y)$.
	The assignment $f\mapsto g$ is called \emph{currying},
	and its inverse is called \emph{uncurrying}.
	(See Theorem~46.11 of~\cite{Munkres}.)
	\item\label{co-curry-htag}
	If $X$ is a locally compact Hausdorff space and $A,B$ are 
	HTAGs, then continuous maps $X\times B\to A$ that are homomorphisms
	in the second variable are in bijection with continuous homomorphisms
	$B\to\C(X,A)$, via the map
	$f\mapsto g$, $g(b)(x) = f(x,b)$.
	This is again called \emph{currying}, and its inverse is
	\emph{uncurrying}.
	(This is an immediate consequence of~\ref{co-curry} above.)
\end{enumerate}

Let us give a typical illustration of how we will use the framework
of function spaces and currying.

\begin{example}\label{example-curry}
	Let $X,Y$ be topological spaces, with $X$ locally compact Hausdorff,
	and let $A$ be an HTAG.
	Then continuous maps $g\colon X\times \bbZ Y\to A$ for which
	all $g(x,-)$ are homomorphisms are in bijection with
	continuous maps $f\colon X\times Y\to A$.
	To see this, observe that by currying, continuous maps 
	$g\colon X\times \bbZ Y\to A$ for which
	each $g(x,-)$ is a homomorphism are in bijection with continuous
	homomorphisms $\bbZ Y\to\C(X,A)$.
	And similarly, again by currying,
	continuous maps $f\colon X\times Y\to A$
	are in bijection with continuous maps $Y\to\C(X,A)$.
	But now continuous homomorphisms $\bbZ Y\to\C(X,A)$ are in bijection
	with continuous maps $Y\to\C(X,A)$ by the universal property of 
	the free HTAG, and this establishes the claim.
\end{example}

\section{Polynomial functions of abelian (semi)groups}
\label{section-general-polynomiality}

In this section we will establish the general theory
of polynomiality that we need in the rest of the paper.
This theory will allow us to recognise and describe `polynomial'
functions that have values in an abelian group, despite the absence
of any notion of multiplication or exponentiation.
The key idea here is to 
recognise polynomiality in terms of the vanishing of iterated differences.
\emph{Understanding} polynomials by means of differences between their values
has a long history, at least as far back as Newton's method of divided 
differences, while \emph{characterising} polynomials by means of differences 
seems to go back to Fr\'echet~\cite{FrechetI,FrechetII}.
The most relevant reference for us is Djokovi\'c's paper~\cite{Djokovic},
and we will recall and extend the work of that paper throughout the section.

\subsection{Polynomial functions on $[0,\infty)$}

We begin by stating the results that will be used
in the remainder of the paper.

\begin{definition}[Polynomial expansions]
\label{polynomial}
	Let $M$ be an abelian group.
	A \emph{polynomial expansion} of a function 
	$f\colon [0,\infty)\to M$ is an expression of the form
	\[
		f(a) = f_0 + f_1(a^1) +\cdots + f_d(a^d),
		\qquad a\in A
	\]
	for some constant $f_0\in A$ 
	and some homomorphisms $f_1,\ldots,f_d\colon [0,\infty)\to M$.
	We call the $f_i$ the \emph{components} of the polynomial
	expansion.
\end{definition}

When they exist, such polynomial expansions are unique:

\begin{proposition}[Uniqueness of polynomial expansions]
\label{uniqueness}
	Let $M$ be an abelian group.
	Suppose that $f\colon [0,\infty)\to M$ has two polynomial expansions
	\[
		f(a) = f_0+f_1(a^1)+\cdots+f_d(a^d)
		\quad\text{and}\quad
		f(a) = g_0+g_1(a^1)+\cdots+g_{d'}(a^{d'}).
	\]
	Then $d=d'$ and $f_i = g_i$ for $i=0,\ldots,d$.
\end{proposition}

Let $M$ be an abelian group.
For $u\in [0,\infty)$ the \emph{difference operator}
$\Delta_u$ takes a function $f\colon [0,\infty)\to M$ into the function
$\Delta_u f\colon [0,\infty)\to M$ defined by
\[
	(\Delta_u f)(a) = f(a+u)-f(a)
\]
for $a\in [0,\infty)$.
Our next result is a criterion for the existence of polynomial
expansions, given in terms of these difference operators.
It includes topologies
\emph{and} parameters, which is essential for the applications,
although it adds some awkwardness to the phrasing.

\begin{theorem}[Existence of polynomial expansions]
\label{existence}
	Let $X$ be a topological space, 
	let $M$ be a topological abelian group,
	and let $f\colon [0,\infty)\times X\to M$ be 
	a continuous function such that the following condition holds
	for all $x\in X$:
	\[
		\Delta_{u_1}\cdots \Delta_{u_{n+1}} f(-,x)=0
		\qquad\text{for\ all\ }
		u_1,\ldots,u_{n+1}\in [0,\infty)
	\]
	Then 
	\[
		f(a,x)= f_0(x) + f_1(a^1,x)+\cdots+f_n(a^n,x)
	\]
	for all	$(a,x)\in A\times X$, where
	$f_0\colon X\to M$ and $f_1,\ldots,f_n\colon [0,\infty)\times X\to M$
	are continuous functions, and each of
	$f_1,\ldots,f_n$ is a homomorphism in its first variable.
\end{theorem}

The remainder of this section is dedicated to the proofs of the above results.

\subsection{Algebraic polynomiality}
We begin with a more general algebraic version of the existence and
uniqueness of polynomial expansions.
Until further notice, let $A$ be an abelian semigroup and $M$ an abelian group.

For $u\in A$ the \emph{difference operator}
$\Delta_u$ takes a function $f\colon A\to M$ into the function
$\Delta_u f\colon A\to M$ defined by
\[
	(\Delta_u f)(a) = f(a+u)-f(a)
\]
for $a\in A$.
The difference operator satisfies the identities
\begin{gather}
	\Delta_u\Delta_v = \Delta_v\Delta_u
	\label{delta-relation-one}
	\\
	\Delta_{u+v}-\Delta_u-\Delta_v = \Delta_u\Delta_v
	\label{delta-relation-two}
\end{gather}
for any $u,v\in A$~\cite[Lemma~2]{Djokovic}.
If $f\colon A^n\to M$ is symmetric, and a homomorphism in each of its variables,
then for any $u_1,\ldots,u_p\in A$ we have:
\begin{equation}\label{equation-iterated-difference}
	\Delta_{u_1}\cdots\Delta_{u_p} f 
	=
	\begin{cases}
		0 & \text{if\ }p>n
		\\
		n! f(u_1,\ldots,u_n) & \text{if\ }p=n
	\end{cases}
\end{equation}
In the second case, the notation indicates that 
$\Delta_{u_1}\cdots\Delta_{u_p} f$ is the constant function
with value $n! f(u_1,\ldots,u_n)$.
See~\cite[pp.193-194]{Djokovic}.

In the present setting we need a generalised notion of polynomial
expansion that we phrase as follows.
Given $f\colon A^n\to M$, the \emph{diagonalisation} of $f$,
denoted $f^\ast\colon A\to M$, 
is defined by
\[
	f^\ast(a)=f(a,\ldots,a)
\]
for $a\in A$.
(When $f\colon A^n\to M$ is a homomorphism in each variable,
we can think of $f^\ast$ as a kind of `monomial of degree $n$'.)

\begin{definition}\label{polynomial-algebraic}
	A \emph{polynomial expansion} of a function 
	$f\colon A\to M$ is an expression of the form
	\[
		f(a) = f_0^\ast(a) + f_1^\ast(a) +\cdots + f_d^\ast(a),
		\qquad a\in A
	\]
	where each $f_i\colon A^i\to M$ is symmetric and a homomorphism
	in each variable.
	(Note that the domain of $f_0$ is $A^0$,
	so that $f_0$ and $f_0^\ast$ are simply constants.)
\end{definition}

The next two results are the algebraic versions of 
Proposition~\ref{uniqueness} and Theorem~\ref{existence}.
In order to take account of the factor $n!$ appearing 
in~\eqref{equation-iterated-difference}, they rely on an
invertibility condition holding for one of $A$ and $M$.
The case where $M$ satisfies the invertibility condition is essentially
Theorem~3 of~\cite{Djokovic}.
The case where $A$ satisfies the invertibility condition
--- this is the case we need for the rest of the paper ---
follows~\cite{Djokovic} closely with some amendments to the details.

\begin{proposition}
\label{uniqueness-algebraic}
	Suppose that one of $A$ and $M$ satisfies the 
	\emph{invertibility condition} that for each $d\geqslant 1$
	the self-map $x\mapsto d!x$ is invertible.
	Suppose that $f\colon A\to M$ has two polynomial expansions
	\[
		f(a) = \sum_{i=0}^d f_i^\ast(a)
		\quad\text{and}\quad
		f(a) = \sum_{j=0}^{d'}g_j^\ast(a).	
	\]
	Then $d=d'$ and $f_i = g_i$ for $i=0,\ldots,d$.
\end{proposition}

\begin{proof}
	It is sufficient to consider the case where the $g_i$ all vanish,
	so that we must show that the $f_i$ also all vanish.
	Observe from \eqref{equation-iterated-difference} that 
	$d! f_d(u_1,\ldots,u_d) = 0$ for all $u_1,\ldots,u_d\in A$.
	When the invertibility condition holds for $M$ we then obtain
	$f_d(u_1,\ldots,u_d)=0$ immediately.
	When the condition holds for $A$
	we can replace $u_1$ with $\frac{u_1}{d!}$,
	where $\frac{u_1}{d!}$ denotes the image of $u_1$ under the inverse
	of $a\mapsto d!a$.
	This gives us
	$f_d(u_1,\ldots,u_d)
	=d! f_d(\frac{u_1}{d!},\ldots,u_d)
	= 0$
	by additivity of $f_d$ in its first variable.
	Thus $f_d=0$, and repeating the process shows that the $f_i$
	all vanish.
\end{proof}

\begin{theorem}\label{existence-algebraic}
	Suppose that one of $A$ and $M$ satisfies the 
	\emph{invertibility condition} that for each $d\geqslant 1$
	the self-map $x\mapsto d!x$ is invertible.
	Then any function $f\colon A\to M$ that satisfies the condition
	\begin{equation}\label{equation-vanishing-condition}
		\Delta_{u_1}\cdots\Delta_{u_{n+1}} f =0
		\text{\ for\ all\ }u_1,\ldots,u_{n+1}\in A
	\end{equation}
	for some $n\geqslant 1$ has a polynomial expansion
	\[
		f = \sum_{k=0}^n f_k^\ast
	\]
	where each $f_k\colon A^k\to M$ is symmetric, 
	and is a homomorphism in each of its variables.
\end{theorem}

\begin{proof}
	The result is proved by induction on $n$,
	the case $n=0$ being immediate.
	Suppose now that $n>0$, and that existence is shown for all smaller
	values of $n$, and suppose given $f$ satisfying
	condition~\eqref{equation-vanishing-condition}.
	Define $h\colon A^n\to M$ by 
	$h(u_1,\ldots,u_n)=\Delta_{u_1}\cdots \Delta_{u_n}f$.
	Condition~\eqref{equation-vanishing-condition}
	means that the right hand side is indeed a constant.
	Equations~\eqref{delta-relation-one} and~\eqref{delta-relation-two}
	can be used to show that $h$ is symmetric, 
	and a homomorphism in each variable, 
	as in the proof of~\cite[Theorem~3]{Djokovic}.
	Define $f_n\colon A^n\to M$ by 
	$f_n(u_1,\ldots,u_n)=\frac{1}{n!}h(u_1,\ldots, u_n)$
	in the case that $M$ satisfies the invertibility assumption,
	and by 
	$f_n(u_1,\ldots,u_n)=(n!)^{n-1}h(\frac{u_1}{n!},\ldots,\frac{u_n}{n!})$
	in the case that $A$ satisfies the invertibility assumption.
	In each case we see that $f_n$ is symmetric, 
	and a homomorphism in each variable,	
	and using~\eqref{equation-iterated-difference} we see that
	$\Delta_{u_1}\cdots \Delta_{u_n} f_n^\ast 
	= \Delta_{u_1}\cdots\Delta_{u_n} f$.
	So now the function $f-f_n^\ast$ 
	satisfies~\eqref{equation-vanishing-condition}
	with $n$ replaced by $n-1$, and so by the inductive hypothesis
	can be written $f - f_n^\ast = \sum_{k=0}^{n-1} f_k^\ast$.
	This completes the proof.
\end{proof}

\subsection{Specialising to $[0,\infty)$}

We are now ready to specialise to the case $A=[0,\infty)$.
In order to get from diagonalisations $f_i^\ast(a)=f_i(a,\ldots,a)$,
which appear in the algebraic results of the last section, 
to functions of the form $f_i(a^i)$, which appear in our main results,
we use the following proposition.

\begin{proposition}\label{proposition-multilinear}
	Let $M$ be a Hausdorff topological abelian group.
	Let $f\colon [0,\infty)^n\to M$ be a function that in each variable
	is a continuous homomorphism.
	Then for any $i\neq j$ 
	and any $a,\lambda_1,\ldots,\lambda_n\in[0,\infty)$,
	we have
	\[
		f(\lambda_1,\ldots,a\lambda_i,\ldots,\lambda_n)
		=
		f(\lambda_1,\ldots,a\lambda_j,\ldots,\lambda_n).
	\]
	Consequently 
	$f(\lambda_1,\ldots,\lambda_n) 
	= 
	g(\lambda_1\cdots \lambda_n)$
	where $g\colon [0,\infty)\to M$ is the homomorphism defined by
	$g(\lambda) = f(\lambda,1,\ldots,1)$.
\end{proposition}

\begin{proof}
	For simplicity take $n=2$.
	If $a$ is a natural number then 
	$f(a\lambda_1,\lambda_2)
	= f(\lambda_1,a\lambda_2)$
	since both are equal to 
	$af(\lambda_1,\lambda_2)$.
	And by applying this property to $\lambda_1/a,\lambda_2/a$ we 
	obtain
	$f(\lambda_1,\lambda_2/a) = f(\lambda_1/a,\lambda_2)$.
	So altogether we obtain the required identity when $a$ is rational.
	Now let $a\in [0,\infty)$ be arbitrary, and let $(a_n)$ be a sequence
	of non-negative rationals converging to $a$.
	Then since $f$ is continuous in each variable, we have
	\begin{align*}
		f(a\lambda_1,\lambda_2)
		&=
		f(\lim_{n\to\infty} a_n \lambda_1,\lambda_2)
		\\
		&=
		\lim_{n\to\infty} f(a_n \lambda_1,\lambda_2)
		\\
		&=
		\lim_{n\to\infty} f(\lambda_1,a_n\lambda_2)
		\\
		&=
		f(\lambda_1,\lim_{n\to\infty} a_n\lambda_2)
		\\
		&=
		f(\lambda_1, a\lambda_2).
	\end{align*}
	Here we have used the fact that $M$ is Hausdorff to ensure that
	limits are uniquely defined.
\end{proof}

\begin{proof}[Proof of Theorem~\ref{uniqueness}]
	Observe that a polynomial expansion 
	$f(a)=f_0+f_1(a^1)+\cdots +f_d(a^d)$
	in the sense of Definition~\ref{polynomial}
	is an expansion $f(a)=g_0^\ast(a)+\cdots+g_d^\ast(a)$
	in the sense of Definition~\ref{polynomial-algebraic},
	where we define $g_i\colon [0,\infty)^i\to M$ by
	$g_i(a_1,\ldots,a_i) = f_i(a_1\cdots a_i)$.
	Thus Theorem~\ref{uniqueness} follows directly from 
	Theorem~\ref{uniqueness-algebraic}.
\end{proof}

\begin{proof}[Proof of Theorem~\ref{existence}]
	Applying Theorem~\ref{existence-algebraic} to the function
	$f(-,x)$ for each $x\in X$, we obtain a polynomial expansion
	in the sense of Definition~\ref{polynomial-algebraic}
	\[
		f(a,x) = \sum_{k=0}^n f_k(a,\ldots,a,x)
	\]
	for $(a,x)\in[0,\infty)\times X$,
	where each $f_k\colon A^k\times X\to M$ is symmetric
	in its first $k$ variables, 
	and a homomorphism in each of its first $k$ variables.

	To see that the $f_i$ are continuous,
	we examine the proof of Theorem~\ref{existence-algebraic}
	and see that $f_n$ is described in terms of the function
	$h(u_1,\ldots,u_n,x)=\Delta_{u_1}\cdots\Delta_{u_n}f(-,x)$,
	which is continuous by inspection, being a $\bbZ$-linear combination
	of $f(-,x)$ evaluated on $\bbZ$-linear combinations of the $u_i$.
	The passage from $h$ to $f_n$ is given by inverting the self-map
	$a\mapsto n!a$, which is a homeomorphism of $[0,\infty)$, 
	and this is sufficient to show that $f_n$ is continuous.
	Replacing $f$ by $f-f^\ast_n$ and repeating this argument,
	it follows that the $f_i$ are all continuous.

	Finally, define $g_1,\ldots,g_n\colon [0,\infty)\times X\to M$
	by $g_i(a,x) = f_i(a,1,\ldots,1,x)$ for $(a,x)\in[0,\infty)\times X$,
	and $g_0=f_0\colon X\to M$.  Then the $g_i$ are all continuous,
	and by Proposition~\ref{proposition-multilinear} we have
	$g_i(a^i,x) = f_i(a,\ldots,a,x)$ for all $(a,x)\in [0,\infty)\times X$,
	and consequently 
	\[
		f(a,x)=g_0(x)+g_1(a^1,x)+\cdots+g_n(a^n,x)
	\]
	is the required polynomial expansion in the sense of 
	Definition~\ref{polynomial}.
\end{proof}

\section{McMullen polynomiality with values in a topological
abelian group}\label{section-McMullen}

In this section we will prove a version of McMullen polynomiality
that applies to continuous translation-invariant valuations with 
values in an arbitrary Hausdorff topological abelian group.
The main result is the following, which shows that the
functions $\lambda\mapsto\varphi(\lambda X)$ 
satisfy the criterion for polynomiality
that appears in Theorem~\ref{existence}, thereby allowing the theory
of section~\ref{section-general-polynomiality} to be applied.

\begin{theorem}\label{mcmullen-polynomiality}
	Let $V$ be a finite dimensional real vector space, 
	let $\varphi\colon\K(V)\to A$ be a continuous translation-invariant
	valuation with values in a Hausdorff topological abelian group $A$,
	and let $X$ be a convex body of dimension $d$ in $V$.
	Then the function $[0,\infty)\to A$, 
	$\lambda\mapsto \varphi(\lambda X)$
	vanishes under $\Delta_{b_1}\cdots\Delta_{b_{d+1}}$
	for any $b_1,\ldots,b_{d+1}\in[0,\infty)$.
\end{theorem}

Theorem~C of the introduction now follows by immediately 
by applying Theorem~\ref{existence} to the function
$\lambda\mapsto \varphi(\lambda X)$.
In later sections we will use Theorems~\ref{uniqueness} and~\ref{existence} 
in more detail to obtain Theorem~B of the introduction.

The main ingredient in the proof of Theorem~\ref{mcmullen-polynomiality}
is Lemma~\ref{simplex-decomposition} below,
which is a version of the canonical simplex decomposition
given by McMullen in~\cite[Lemma~10]{McMullenPolytope},
and elaborates on the decomposition used by Chen in~\cite[Lemma~3.4]{Chen}.

Let $V$ be a vector space, let $v_1,\ldots,v_d$ be linearly independent
vectors in $V$, and define
\[
	\sigma^d = S(v_1,\ldots,v_d)
	=
	\left\{
		x_1v_1+\cdots+x_dv_d
		\mid
		1\geqslant x_1\geqslant\cdots\geqslant x_d\geqslant 0
	\right\}.
\]
This is the $d$-simplex with vertices $0,v_1,v_1+v_2,\ldots,v_1+\cdots+v_d$.
Any $d$-simplex with $0$ as a vertex has this form for an appropriate choice
of $v_1,\ldots,v_d$.

\begin{lemma}[A simplex decomposition]
	\label{simplex-decomposition}
	Let $V$ be a vector space, let $v_1,\ldots,v_d\in V$ 
	be linearly independent, and let 
	\[
		\sigma^i = S(v_1,\ldots,v_i),
		\qquad
		\tau^{d-i} = S(v_{i+1},\ldots,v_d)
	\]
	for $i=0,\ldots,d$.
	Let $a,b>0$, and define
	\begin{align*}
		A_i &= a\sigma^{i\phantom{-1}} 
		+ b\tau^{d-i} + b(v_1+\cdots +v_i)
		\quad \text{for}\ 
		 i\geqslant 0,
		\\
		B_i &= a\sigma^{i-1} 
		+ b\tau^{d-i} + b(v_1+\cdots+ v_i)
		\quad \text{for}\ 
		i\geqslant 1.
	\end{align*}
	Then 
	\begin{equation}\label{eqn-overall}
		(a+b)\sigma^d
		=
		A_0\cup\cdots\cup A_d
	\end{equation}
	and
	\begin{equation}\label{eqn-step}
		(A_0\cup\cdots\cup A_{i-1})\cap A_i = B_i
		\quad \text{for}\ 
		i\geqslant 1
	\end{equation}
	Consequently, if $\varphi$ is a translation-invariant
	valuation on $\K(V)$, then we have
	\begin{equation}\label{eqn-final}
		\varphi((a+b)\sigma^d)
		=
		\sum_{i=0}^d \varphi(a\sigma^i + b\tau^{d-i})
		-
		\sum_{i=1}^{d}\varphi (a\sigma^{i-1}+b\tau^{d-i}).
	\end{equation}
\end{lemma}

\begin{proof}
	One can check that 
	\begin{align*}
		A_i
		&= 
		\{x_1v_1+\cdots+x_dv_d\in(a+b)\sigma
		\mid x_i\geqslant b\geqslant x_{i+1}\},
		\\
		B_i
		&= 
		\{x_1v_1+\cdots+x_dv_d\in(a+b)\sigma
		\mid x_i= b\}
	\end{align*}
	where, in the cases $i=0,d$, the inequalities involving $x_0,x_{d+1}$
	are omitted.
	Consequently
	\[
		A_0\cup\cdots\cup A_{i-1}
		= 
		\{x_1v_1+\cdots+x_dv_d\in(a+b)\sigma^d\mid b\geqslant x_i\},
	\]
	with the inequality omitted in the case $i-1=d$.	
	Equations~\eqref{eqn-overall} and~\eqref{eqn-step} now follow.
	Finally we use \eqref{eqn-overall} and~\eqref{eqn-step}
	and the fact that $\varphi$ is a valuation to show that
	$\varphi((a+b)\sigma^d)
	=
	\sum_{i=0}^d\varphi(A_i) - \sum_{i=1}^d \varphi(B_i)$.
	The definition of the $A_i$ and $B_i$, together with 
	translation-invariance of $\varphi$, now give~\eqref{eqn-final}.
\end{proof}

\begin{proof}[Proof of Theorem~\ref{mcmullen-polynomiality}]
	Let us write $f$ for the function
	$f\colon [0,\infty)\to A$, $\lambda\mapsto \varphi(\lambda X)$.
	We prove the theorem by induction on the dimension $d$.
	In the case $d=0$ the function $f$
	is constant, so that it vanishes under $\Delta_{b_1}$ as required.
	So let us assume that $d>0$ and that the result holds for all
	smaller values of $d$.
	Since $\Delta_{b_1}\cdots\Delta_{b_{d+1}} f (\lambda_0)$
	depends continuously on $X$, we may assume that $X$ is a polytope.
	And since $\varphi$ is a translation-invariant valuation,
	we may in fact assume that $X$ is a simplex with a vertex at the
	origin.
	Write $X$ as $\sigma^d = S(v_1,\ldots,v_d)$ for appropriate
	$v_1,\ldots,v_d\in V$.
	Then by Lemma~\ref{simplex-decomposition} we have:
	\begin{align*}
		\Delta_{b_{d+1}}f (a)
		&=
		f(a+b_{d+1})-f(a)
		\\
		&=
		\sum_{i=0}^d \varphi(a\sigma^i + b_{d+1}\tau^{d-i})
		-
		\sum_{i=1}^{d}\varphi (a\sigma^{i-1}+b_{d+1}\tau^{d-i})
		-\varphi(a\sigma^d)
		\\
		&=
		\sum_{i=0}^{d-1} \varphi(a\sigma^i + b_{d+1}\tau^{d-i})
		-
		\sum_{i=1}^{d}\varphi (a\sigma^{i-1}+b_{d+1}\tau^{d-i})
	\end{align*}
	Observe that every term in this expression has form
	$\varphi(a\sigma^i+Y)$ for some $i$ in the range 
	$0\leqslant i\leqslant d-1$, where $\sigma^i$ and $Y$ lie in
	subspaces of $V$ which have trivial intersection.
	Let us fix such a term $\varphi(a\sigma^i+Y)$,
	let us write $W$ for the affine subspace spanned by $\sigma^i$,
	and note that the assignment $Z\mapsto \varphi(Z+Y)$ is a continuous
	translation-invariant valuation on $\K(W)$.
	It then follows by the inductive hypothesis
	that $Z\mapsto \varphi(Z+Y)$ vanishes under 
	$\Delta_{b_1}\cdots\Delta_{b_d}$ for any $b_1,\ldots,b_d\in[0,\infty)$.
	Consequently $\Delta_{b_1}\cdots\Delta_{b_d}(\Delta_{b_{d+1}} f)$
	vanishes, as required.
\end{proof}

\section{Groups of convex bodies}
\label{section-grothendieck}

In this section we introduce the group $\CB(V,G)$ of convex bodies
in $V$, and the universal valuation with target $\CB(V,G)$,
and we give examples of how to work with these using the universal property.

\begin{definition}
	Let $V$ be a finite-dimensional real vector space, and let
	$G$ be a group of affine-linear transformations of $V$.
	The \emph{group of convex bodies in $V$},
	denoted $\CB(V,G)$, is the quotient of the 
	free Hausdorff topological group 
	$\bbZ \K(V)$ on $\K(V)$ by the closed subgroup
	generated by all elements of the form
	\[
		[A\cup B]-[A]-[B]+[A\cap B]	
	\]
	for $A,B\in\K(V)$ such that $A\cup B\in\K(V)$,
	\emph{and} all elements of the form
	\[
		[A] - [gA]
	\]
	for $A\in\K(V)$ and $g\in G$.
	The \emph{universal continuous $G$-invariant valuation on $\K(V)$},
	denoted
	\[
		\Phi\colon\K(V)\longrightarrow\CB(V,G),
	\]
	is the composite $\K(V)\to\bbZ\K(V)\to\CB(V,G)$,
	and we continue to denote it by $\Phi(X) = [X]$ for $X\in\K(V)$.
\end{definition}

\begin{proposition}[The universal property of $\Phi$]\label{universal}
	Let $V$ be a finite-dimensional real vector space, and let
	$G$ be a group of affine-linear transformations of $V$.
	Then $\Phi\colon \K(V)\to\CB(V,G)$ 
	is a continuous $G$-invariant valuation,
	and it has the following universal property:
	If $A$ is a Hausdorff topological abelian group,
	then there is a bijection
	between the set of continuous homomorphisms 
	$\CB(V,G)\to A$ and the set of continuous
	$G$-invariant valuations $\K(V)\to A$.
	This bijection sends a homomorphism $\CB(V,G)\to A$ to 
	its composite with $\Phi$.
\end{proposition}

\begin{proof}
	Let $C\subset \bbZ\K(V)$ denote the subgroup generated by the
	elements listed in the definition of $\CB(V,G)$,
	so that $\CB(V,G) = \bbZ\K(V) / \bar C$.
	Continuous homomorphisms $\varphi\colon\CB(V,G)\to A$ are the same
	thing as continuous homomorphisms $\bbZ\K(V)\to A$ that send
	the generators of $C$ to $0$,
	and by the universal property of $\bbZ\K(V)$
	(Definition~\ref{free-HTAG}) these are the same as
	continuous maps $\K(V)\to A$ that satisfy the valuation and
	$G$-invariance properties.
\end{proof}

The injectivity part of the bijection in Proposition~\ref{universal}
can be phrased simply as follows.
(Most of our applications of the universal property will be of this
part.)

\begin{corollary}\label{injective}
	Let $f,g\colon\CB(V,G)\to A$ be two continuous homomorphisms
	into the same Hausdorff topological abelian group $A$.
	If $f([X])=g([X])$ for all $X\in\K(V)$, then $f=g$.
\end{corollary}

Let us use the universal property to construct a family of maps which
will be important throughout the rest of the paper.
We cover this in detail to try to emphasise how the machinery works.

\begin{example}[Dilation maps]\label{dilation-example}
	Let $V$ be a finite-dimensional real vector space.

	To begin with let us take $\lambda\in[0,\infty)$.
	Then dilation by $\lambda$ determines a map
	$\K(V)\to \K(V)$, $X\mapsto\lambda X$.  
	This map is continuous, $G$-invariant,
	and preserves unions and intersections.
	Composing with $\Phi$, we therefore obtain a 
	continuous $G$-invariant valuation
	$\K(V)\to\CB(V,G)$ defined by $X\mapsto[\lambda X]$.
	Finally, by the universal property of $\Phi$,
	this corresponds to a continuous homomorphism
	$\CB(V,G)\to\CB(V,G)$ for which $[X]\mapsto[\lambda X]$.

	As $\lambda$ varies, the maps just constructed assemble into 
	a \emph{dilation map}
	\[
		D\colon[0,\infty)\times\CB(V,G)\to\CB(V,G)
	\]
	for which
	$(\lambda,[X])\mapsto[\lambda X]$.
	We want this map to be continuous; this is not immediate from
	the previous paragraph, but we can prove it using
	currying.
	The map $[0,\infty)\times\K(V)\to\K(V)$, $(\lambda,X)\mapsto\lambda X$
	is continuous and, by composing with 
	$\Phi$, we obtain a continuous map 
	$[0,\infty)\times\K(V)\to\CB(V,G)$, $(\lambda,X)\mapsto [\lambda X]$.
	By currying we obtain a continuous map 
	$\K(V)\to\C([0,\infty),\CB(V,G))$, 
	$X\mapsto(\lambda\mapsto[\lambda X])$
	that is a $G$-invariant valuation, as in the last paragraph.
	By the universal property of $\Phi$, this determines
	a continuous homomorphism $\CB(V,G)\to\C([0,\infty),\CB(V,G))$
	for which $[X]\mapsto(\lambda\mapsto[\lambda X])$.
	Finally, by uncurrying we obtain the required continuous map
	$D\colon [0,\infty)\times\CB(V,G)\to\CB(V,G)$.

	Let us establish two more properties of $D$.
	First,
	\[
		D(1,x) = x
		\quad\text{for}\quad x\in\CB(V,G).
	\]
	Indeed, $D(1,-)$ sends $[X]$ to $[1X] = [X]$ for any $X\in\K(V)$, 
	and the identity map has the same property,
	and the two must coincide.
	Second,
	\[
		D(\lambda,D(\mu,x))=D(\lambda\mu,x)
		\quad\text{for}\quad
		\lambda,\mu\in[0,\infty),\ x\in\CB(V,G).
	\]
	Indeed, the dilation map $[0,\infty)\times\K(V)\to\K(V)$
	satisfies $\lambda(\mu X) = (\lambda\mu)X$, and so the two maps
	$D(\lambda,D(\mu,-)),D(\lambda\mu,-)$
	send $[X]$ to $[\lambda(\mu X)] = [(\lambda\mu)X]$
	for any $X\in\K(V)$, and they therefore coincide.
\end{example}

\section{The structure of the convex body groups}\label{section-structure}

In this section we use polynomiality to investigate the structure
of the groups $\CB(V,G)$, and we give the proof of Theorems~A and~B.

Let $V$ be a finite dimensional real vector space and let $G$ be
a group of affine-linear transformations of $V$.
Recall the dilation map
\[
	D\colon [0,\infty)\times\CB(V,G)\to\CB(V,G)
\]
from Example~\ref{dilation-example}, which is characterised by the fact that
$D(\lambda,[X]) = [\lambda X]$.

\begin{proposition}\label{criterion}
	Let $V$ be a real vector space of dimension $d$,
	and let $G$ be
	a group of affine-linear transformations of $V$
	that includes all translations.
	Then the dilation map
	\[
	D\colon [0,\infty)\times\CB(V,G)\to\CB(V,G)
	\]
	of Example~\ref{dilation-example}
	satisfies
	\begin{equation}\label{vanishing}
		\Delta_{u_1}\cdots\Delta_{u_{d+1}} D(-,x)=0
	\end{equation}
	for any $x\in\CB(V,G)$ and any $u_1,\ldots,u_{d+1}\in[0,\infty)$.
\end{proposition}

\begin{proof}
	Applying Theorem~\ref{mcmullen-polynomiality} to the universal
	valuation $\Phi$ shows that, for any $X\in\K(V)$,
	the map $\lambda\mapsto \Phi(\lambda X) = [\lambda X]$
	vanishes under 
	$\Delta_{u_1}\cdots\Delta_{u_{m+1}}$
	for any $u_1,\ldots,u_{m+1}\in[0,\infty)$,
	where $m$ is the dimension of $X$.
	It then follows that $\lambda\mapsto [\lambda X]$
	vanishes under 
	$\Delta_{u_1}\cdots\Delta_{u_{d+1}}$
	for any $u_1,\ldots,u_{d+1}\in[0,\infty)$.
	Since $[\lambda X]=D(\lambda,[X])$, it follows that~\eqref{vanishing}
	holds in the case $x=[X]$.

	Now fix $\lambda\in[0,\infty)$ and consider the map
	$\CB(V,G)\to\CB(V,G)$, that sends $x\in\CB(V,G)$ to 
	$\Delta_{u_1}\cdots\Delta_{u_{d+1}} D(-,x)(\lambda)$.
	This is a continuous homomorphism, and by the last paragraph
	it vanishes whenever $x=[X]$ for $X\in\K(V)$,
	so by Corollary~\ref{injective} it vanishes identically, as required.
\end{proof}

The last proposition allows us to 
apply the existence of polynomial expansions (Theorem~\ref{existence})
to $D$ to obtain the following definition.

\begin{definition}[The dilation components]\label{dilation-components}
	Let $V$ be a real vector space of dimension $d$,
	and let $G$ be
	a group of affine-linear transformations of $V$
	that includes all translations.
	We define the \emph{dilation components}
	\[
		\rho_0\colon\CB(V,G)\longrightarrow\CB(V,G)
		\quad\text{and}\quad
		\rho_1,\ldots,\rho_d\colon
		[0,\infty)\times \CB(V,G)\longrightarrow\CB(V,G)
	\]
	to be the components of the polynomial expansion of
	the dilation map $D$.
	Thus the $\rho_i$ are uniquely determined by the fact that
	\begin{equation}\label{homogeneous-expansion}
		D(\lambda,x)
		=
		\rho_0(x)
		+
		\rho_1(\lambda^1,x)
		+\cdots+
		\rho_d(\lambda^d,x)
	\end{equation}
	for all $\lambda\in [0,\infty)$ and $x\in\CB(V,G)$,
	and by the fact that $\rho_i(-,x)$ is a homomorphism for each
	$i=1,\ldots,d$ and $x\in\CB(V,G)$.
\end{definition}

In what follows we will abuse notation by writing 
$\rho_i(\lambda^i,x)$ for $i=0,\ldots,d$,
including the case $i=0$,
thus ignoring the fact that the domain of $\rho_0$ is $\CB(V,G)$ rather
than $[0,\infty)\times\CB(V,G)$.
We hope that this will not cause the reader too much consternation;
it will certainly streamline our proofs in many places.

\begin{proposition}\label{rhoihomomorphism}
	In the situation of Proposition~\ref{criterion},
	the map $\rho_0$ is a homomorphism, and the maps
	$\rho_1,\ldots,\rho_d$ are homomorphisms in their second variable.
\end{proposition}

\begin{proof}
	The dilation map $D\colon[0,\infty)\times\CB(V,G)
	\to \CB(V,G)$ is a homomorphism in its second variable.
	Fixing $x,y\in\CB(V,G)$, we have 
	$D(\lambda,x+y) = D(\lambda,x) + D(\lambda,y)$
	for all $\lambda\in[0,\infty)$.
	Applying~\eqref{homogeneous-expansion} to each term of this
	equation and rearranging gives us
	\[
		\sum_{i=0}^d \rho_i(\lambda^i,x+y)
		=
		\sum_{i=1}^d \left(
			\rho_i(\lambda^i,x)+\rho_i(\lambda^i,y)
		\right)
	\]
	so that we have two polynomial expansions of the same function.
	(Observe that $\rho_i(-,x)$ and $\rho_i(-,y)$ are homomorphisms,
	so that the same is true of $\rho_i(-,x)+\rho_i(-,y)$.)
	Applying uniqueness (Proposition~\ref{uniqueness}) 
	gives us the result.
\end{proof}

\begin{proposition}\label{pre-idempotence}
	In the situation of Proposition~\ref{criterion},
	the dilation components satisfy the relations
	\[
		\rho_i(\lambda,\rho_i(\mu,x))
		=
		\rho_i(\lambda\mu,x)
		\qquad\text{and}\qquad
		\rho_i(\lambda,\rho_j(\mu,x))
		=
		0
		\ \text{for\ }i\neq j
	\]
	for all $\lambda,\mu\in[0,\infty)$ and all $x\in\CB(V,G)$.
\end{proposition}

\begin{proof}
	The dilation map satisfies $D(\lambda,D(\mu,x))=D(\lambda\mu,x)$ 
	for all $\lambda,\mu\in[0,\infty)$ and all $x\in\CB(V,G)$ so  that
	by applying~\eqref{homogeneous-expansion} three times we obtain
	the identity:
	\[
		\sum_{i,j=0}^d\rho_i(\lambda^i,\rho_j(\mu^j,x))
		=
		\sum_{k=0}^d\rho_k(\lambda^k\mu^k,x)
	\]
	Fixing $\mu$ and $x$ this identity relates two polynomial expansions
	in the variable $\lambda$, so that by uniqueness of polynomial
	expansions (Proposition~\ref{uniqueness}) we have
	\[
		\sum_{j=0}^d\rho_i(\lambda^i,\rho_j(\mu^j,x))
		=
		\rho_i(\lambda^i\mu^i,x)
	\]
	for each $i$.
	Fixing $\lambda$ and $x$, this identity relates two polynomial
	expansions in the variable $\mu$, and so by applying uniqueness
	once more we obtain the result.
\end{proof}

\begin{definition}[The McMullen idempotents]
\label{idempotents-splitting}
	In the situation of Proposition~\ref{criterion},
	we define the \emph{McMullen idempotents} to be the homomorphisms
	\[
		e_0,\ldots,e_d\colon\CB(V,G)\longrightarrow\CB(V,G)
	\]
	defined by $e_0=\rho_0$ and 
	$e_i(x)=\rho_i(1,x)$ for $i=1,\ldots,d$ and $x\in\CB(V,G)$.
	By equation~\eqref{homogeneous-expansion} we have
	\[
		e_0+\cdots+e_d = 1
	\]
	since $D(1,-)$ is the identity map.
	And by Proposition~\ref{pre-idempotence} we have
	\[
		e_i^2=e_i,
		\qquad
		e_ie_j=0
	\]
	for $i,j=0,\ldots,d$ and $i\neq j$.
	Thus the McMullen idempotents 
	$e_i$ form a complete set of orthogonal idempotents.
	They therefore express $\CB(V,G)$ as a direct sum of closed subgroups
	\[
		\CB(V,G)
		=
		\CB_0(V,G)\oplus\cdots\oplus\CB_d(V,G)
	\]
	where $\CB_i(V,G) = e_i(\CB(V,G))=\ker(1-e_i)$ for each $i$.
	We refer to this as the \emph{McMullen decomposition}
	of $\CB(V,G)$.
	With respect to this decomposition we write the universal valuation
	$\Phi\colon\K(V)\to\CB(V,G)$ as 
	\[
		\Phi = \Phi_0\oplus\cdots\oplus\Phi_d.
	\]
\end{definition}

Note that if $X$ is a convex body in $V$, then $[X]\in\CB(V,G)$ will not 
necessarily lie in a single group $\CB_i(V,G)$.
Rather, it will typically have nonzero components 
in $\CB_i(V,G)$ for each $i$ up to and including
the dimension of $X$.

We now come to the main theorem of the section.

\begin{theorem}\label{theorem-a-refined}
	\begin{enumerate}
		\item
		$\CB_0(V,G)$ is a copy of $\bbZ$ generated by 
		the class of any
		one-point convex body.
		\item
		For $i>0$, each $\CB_i(V,G)$
		admits the structure of a 
		Hausdorff topological vector space over $\bbR$, extending its 
		structure as a topological abelian group.
		With respect to this vector space structure, the valuation
		$\Phi_i\colon\K(V)\to\CB_i(V,G)$ is homogeneous of degree $i$.
	\end{enumerate}
\end{theorem}

\begin{lemma}\label{one-point}
	Let $p=\{0\}$ denote the one-point convex body consisting of the
	origin alone.
	Then $[p]\in\CB_0(V,G)$ and 
	\[
		\Phi_0(X)=e_0([X]) = [p]
	\]
	for any $X\in\K(V)$.
\end{lemma}

\begin{proof}
	$\Phi_0(X)=e_0([X])$ by the definition of $\Phi_0$.
	And to prove that $e_0([X])=[p]$
	we set $\lambda=0$ in~\eqref{homogeneous-expansion}
\end{proof}

\begin{lemma}\label{rhoi-ei}
	The maps $e_i$ and $\rho_i$ are compatible in the following sense:
	\[
		e_i(\rho_j(\lambda,x))
		=
		\rho_i(\lambda, e_j(x))
		=
		\begin{cases}
			\rho_i(\lambda,x) 
			& \text{if\ }i=j
			\\
			0
			& \text{if\ }i\neq j
		\end{cases}
	\]
	In particular, $\rho_i$ takes values in $\CB_i(V,G)$.
\end{lemma}

\begin{proof}
	This follows from Proposition~\ref{pre-idempotence} 
	by specialising one or other of the scaling factors that appear
	there to $1$.
\end{proof}

\begin{proof}[Proof of Theorem~\ref{theorem-a-refined}]
	Let us again write $p=\{0\}$ for the convex body in $V$ consisting of 
	just the origin.
	To prove the first part we define maps 
	\[
		\beta\colon\bbZ\to\CB_0(V,G),
		\qquad
		\gamma\colon\CB_0(V,G)\to\bbZ
	\]
	as follows.
	Define $\beta$ by the specification $\beta(1)=[p]$.
	For $\gamma$, consider the 
	constant valuation $\varphi_0\colon\K(V)\to\bbZ$ with value $1$.
	This corresponds to a homomorphism 
	$\bar\varphi_0\colon\CB(V,G)\to \bbZ$,
	and we define $\gamma$ to be the restriction of $\bar\varphi_0$
	to $\CB_0(V,G)$.
	We will show that $\beta$ and $\gamma$ are inverse isomorphisms.
	
	Let $\pi_0\colon\CB(V,G)\to\CB_0(V,G)$ denote the projection
	map, so that $\pi_0$ is just obtained from 
	$e_0\colon\CB(V,G)\to\CB(V,G)$
	by restricting the codomain. 
	In particular $\pi_0[X] = e_0[X] = [p]$ for any $X\in\K(V)$.
	Now $\beta\gamma \pi_0 = \beta\bar\varphi_0 e_0$ by the definition
	of $\gamma$.
	Evaluating on $[X]$ gives us 
	\[\beta\gamma \pi_0([X]) 
	=\beta\bar\varphi_0 e_0([X])
	=\beta\bar\varphi_0[p]
	=\beta(1)
	=[p]
	=\pi_0[X].\]
	Thus the maps $\beta\gamma\pi_0$ and $\pi_0$
	coincide on elements $[X]$ for $X\in\K(V)$, and are therefore equal
	by the universal property.
	Since the image of $\pi_0$ is $\CB_0(V,G)$, it follows that
	$\beta\gamma$ is the identity map as required.
	For the other composition, $\gamma\beta(1)=\gamma([p])
	=\bar\varphi_0([p])=1$
	so that $\gamma\beta$ is also the identity map.

	For the second part, we consider the continuous map $\ast$
	defined as follows:
	\[
		\ast\colon [0,\infty)\times\CB_i(V,G)\longrightarrow\CB_i(V,G),
		\qquad
		(\lambda,x)\longmapsto \lambda\ast x = \rho_i(\lambda,x)
	\]
	(This takes values in $\CB_i(V,G)$ by Lemma~\ref{rhoi-ei}.)
	It satisfies the following properties
	for all $\lambda,\mu\in[0,\infty)$ and all $x,y\in\CB_i(V,G)$:
	\begin{enumerate}
		\item
		$(\lambda+\mu)\ast x = \lambda\ast x + \mu\ast x$,
		since $\rho_i$ is a homomorphism in its first variable.
		\item
		$\lambda\ast(x+y) = \lambda\ast x + \lambda\ast y$
		by Proposition~\ref{rhoihomomorphism}.
		\item
		$\lambda\ast(\mu\ast x)= \lambda\mu\ast x$
		by Proposition~\ref{pre-idempotence}.
		\item
		$1\ast x = x$ since $1\ast x=\rho_i(1,x)=e_i(x)=x$.
		Here $e_i(x)=x$ because $x\in\CB_i(V,G)$.
	\end{enumerate}
	We may extend $\ast$ to a map 
	\[
		\star
		\colon
		\bbR\times\CB_i(V,G)\to\CB_i(V,G)
	\]
	defined as follows:
	\[
		\lambda\star x
		=
		\begin{cases}
			\lambda\ast x
			& 
			\text{if\ }\lambda\geqslant 0
			\\
			-((-\lambda)\ast x)
			& 
			\text{if\ }\lambda\leqslant 0
		\end{cases}
	\]
	Put another way, if $\lambda\in[0,\infty)$, then
	$(\pm\lambda)\star x = \pm (\lambda\ast x)$.
	From the points above it follows that $0\ast x=0$ for any $x$,
	and consequently $\star$ is well-defined,
	and moreover $\star$ is continuous by the gluing lemma.
	The four points above continue to hold for $\star$ in place
	of $\ast$ by a tedious but straightforward exercise.
	But this now shows precisely that $\star$ determines 
	the structure of a real topological 
	vector space on $\CB_i(V,G)$.
	Since we already know that each $\CB_i(V,G)$ is Hausdorff,
	this completes the proof that the $\CB_i(V,G)$ are Hausdorff
	topological vector spaces.

	To prove homogeneity note that 
	$\Phi_i(\lambda X)$ is the component of $\Phi(\lambda X)$ that
	lies in $\CB_i(V,G)$, and that $\Phi_i(X)
	=e_i\Phi(X)$.
	Now $\Phi(\lambda X) = D(\lambda,[X])
	=\rho_0([X])+\rho_1(\lambda^1,[X])+\cdots+\rho_d(\lambda^d,[X])$
	so that 
	\[
		\Phi_i(\lambda X) 
		= \rho_i(\lambda^i,[X])
		=\rho_i(\lambda^i,e_i[X])
		=\rho_i(\lambda^i,\Phi_i(X))
		=\lambda^i\star\Phi_i(X)
	\]
	as required.
	Here the equality $\rho_i(\lambda^i,[X]) =\rho_i(\lambda^i,e_i[X])$
	follows from Lemma~\ref{rhoi-ei}.
\end{proof}

\begin{proof}[Proof of Theorem~B]
The result follows by combining the decomposition obtained in 
Definition~\ref{idempotents-splitting}, 
with Lemma~\ref{one-point} and Theorem~\ref{theorem-a-refined}.
\end{proof}

\begin{proof}[Proof of Theorem~A]
	Let $p=\{0\}$ be the one-point space at the origin,
	and let $X\in\K(V)$.
	Under dilation we have $0\cdot X = p$.
	By homogeneity we therefore have 
	$\Phi_i(p)=\Phi_i(0\cdot X)=0^i\Phi_i(X)=0$ for $i>0$.
	An arbitrary $G$-invariant continuous valuation 
	$\varphi\colon\K(V)\to A$
	has
	$\varphi = \bar\varphi\circ\Phi
	=\sum_{i=0}^d \bar\varphi\circ\Phi_i$
	so that $c=\varphi(p)=\bar\varphi\Phi_0(p)$.
	Consequently $\varphi-c
	= \bar\varphi\circ \left(\sum_{i=1}^d\Phi_i\right)$
	factors through the Hausdorff topological vector space 
	$\bigoplus_{i=1}^d\CB_i(V,G)$ as required.
\end{proof}
	
\bibliographystyle{plain}
\bibliography{convexbodies}
\end{document}